\documentclass[graybox,envcountsame]{svmult}

\usepackage{mathptmx}       
\usepackage{helvet}         
\usepackage{courier}        
\usepackage{type1cm}        
\usepackage{makeidx}         
\usepackage{graphicx}        
\usepackage{multicol}        
\usepackage[bottom]{footmisc}

\makeindex             

\usepackage{mathtools}
\usepackage{amsmath}
\usepackage{amssymb}
\usepackage{booktabs}

\DeclareMathAlphabet{\mathcal}{OMS}{cmsy}{m}{n}
\newcommand{\N}{\mathbb{N}}
\newcommand{\R}{\mathbb{R}}
\newcommand{\RT}{\mathcal{R\!T}}
\newcommand{\U}{\mathcal{U}}
\newcommand{\V}{\mathcal{V}}
\newcommand{\PS}{\mathcal{P}}
\newcommand{\QS}{\mathcal{Q}}
\newcommand{\vx}{\boldsymbol{x}}
\newcommand{\vn}{\boldsymbol{n}}
\DeclareMathOperator{\dom}{dom}
\DeclareMathOperator{\ran}{ran}
\DeclareMathOperator{\dive}{div}
\DeclareMathOperator{\grad}{grad}
\renewcommand{\ker}{\operatorname{ker}}
\newcommand{\gradp}{\operatorname{{\grad}_\#}}
\newcommand{\divep}{\operatorname{{\dive}_\#}}

\newcommand{\e}{\mathrm{e}}
\newcommand{\tmi}{t_{m,i}}
\newcommand{\Qmr}[1]{Q_m\left[#1\right]_\rho}
\newcommand{\jump}[1]{[\hspace*{-2pt}[#1]\hspace*{-2pt}]}
\newcommand{\scp}[1]{\langle #1 \rangle}
\newcommand{\scprm}[1]{\langle #1 \rangle_{\rho,m}}

\begin{document}

\title*{Homogenisation of parabolic/hyperbolic media}
\author{Sebastian Franz \and Marcus Waurick}
\institute{Marcus Waurick \at University of Strathclyde, \email{marcus.waurick@strath.ac.uk}
\and Sebastian Franz \at Institute of Scientific Computing, TU Dresden \email{sebastian.franz@tu-dresden.de}}
%
%
\maketitle

\abstract{We consider an evolutionary problem with rapidly oscillating
coefficients. This causes the problem to change frequently between a parabolic
and an hyperbolic state. We prove convergence of the homogenisation process
in the unit square and present a numerical method to deal with approximations
of the resulting equations. A numerical study finalises the contribution.}

\smallskip


\section{Introduction}\label{sec:intro}

In the present article, we discuss an academic example of a partial differential equation with highly oscillatory change of type. In real-world applications this change of type can be observed, when discussing a solid-fluid interaction model. In these kind of models, the solid is modelled by a (hyperbolic) elasticity equation and the fluid is of parabolic type. 

An example of the equations to be studied is the following system of equations in the unit square $\Omega=(0,1)^2$
\begin{align*}
     \partial_t u - \Delta u &= f_\textnormal{w} \text{ on }\Omega_\textnormal{w}, \\
     \partial_t^2 u -\Delta u & =f_\textnormal{b} \text{ on }\Omega_\textnormal{b}.
\end{align*}
$\Omega$ should be thought of being a chessboard like structure with $\Omega_\textnormal{w}$ being the white areas and $\Omega_\textnormal{b}$ being the black areas. $u$ satisfies natural transmission conditions on the interfaces. Our aim is to study the limit of the white and black squares' diameters tending to zero. 

We shall present a convergence estimate for this homogenisation problem as well as a numerical study. Equations with change of type (ranging from elliptic to parabolic to hyperbolic) can be treated with the notion of so-called evolutionary equations, which are due to Picard \cite{Picard2009}, see also \cite{Picard2011}. The notion of evolutionary equations is an abstract class of equations formulated in a Hilbert space setting and comprises partial differential-algebraic problems and may further be described as implicit evolution equation, hence the name `evolutionary equations'. 

More precisely, given a Hilbert space $H$ and bounded linear operator $M_0,M_1\in L(H)$ as well as a skew-self-adjoint operator $A$ in $H$, we consider the problem of finding $U\colon \mathbb{R}\to H$ for some given right-hand side $F\colon \R\to H$ such that
\begin{gather}\label{eq:problem}
  (\partial_t M_0+M_1+A)U=F,
\end{gather}
where $\partial_t$ denotes the time derivative. The solution theory for this equation is set up in an exponentially weighted Hilbert space describing space-time. We shall specify the ingredients in the next section. 

For evolutionary equations, a numerical framework has been developed in \cite{FrTW16}. In particular, this numerical treatment allows for equations with change of type. 

Qualitatively, problems with highly oscillatory change of type (varying in between elliptic/parabolic/hyperbolic) have been considered in \cite{W16_SH} in a one-dimensional setting. For a higher dimensional setting of highly-oscillatory type in the context of Maxwell's equations, we refer to \cite{Waurick2018}. For a solid-fluid interaction homogenisation problem with oscillations between hyperbolic and parabolic parts we refer to \cite{Dasser1995}. 

A quantitative result for equations with change of type has been obtained in \cite{Cherednichenko2018,FrW17}. In the latter reference, we have employed results and techniques stemming from \cite{CW17} to transfer operator-norm estimates on (static) problems posed on $\mathbb{R}^n$ to corresponding estimates for periodic time-dependent problems on the one-dimensional unit cell.

The present contribution is very much in line with the approach presented in \cite{FrW17}. The major difference, however, is the transference to a higher-dimensional setting.

For the sake of the argument, we restrict ourselves to two spatial dimensions. The higher-dimensional case is then adopted without further difficulties. 

We shortly comment on the organisation of this paper. We start by presenting the analytical background in the next section. In this section, we shall also derive the necessary convergence estimates for the homogenisation problem.

Our numerical approach will be provided in Section \ref{sec:numerics}. We conclude the article with a small case study.

\section{Analytical background}\label{sec:analysis}

In this section, we rephrase and summarise some results from \cite{CW17}. The key ingredients are \cite[Theorem 3.9]{CW17} as well as \cite[Proposition 3.16]{CW17}. 

First of all, we properly define the operators involved. Let $\Omega=(0,1)^2$. Then we define
\[
   \tilde{\grad} \colon C_\#^1(\Omega)\subseteq L^2(\Omega)\to L^2(\Omega)^2, \phi\mapsto (\partial_j\phi)_{j\in\{1,2\}},
\]
where $C_\#^1(\Omega)\coloneqq \{\phi|_\Omega; \phi\in C^1(\mathbb{R}^2), \phi(\cdot)=\phi(\cdot+k)\quad (k\in\mathbb{Z}^2)\}$. Note that $\tilde{\grad}$ is densely defined as $C_c^1(\Omega)\subseteq C_\#^1(\Omega)$. We define $\dive_\#\coloneqq -\tilde{\grad}^*$. It is easy to see, that $C^1_\#(\Omega)^2\subseteq \dom(\dive_\#)$ and so $\grad_\#\coloneqq -\dive_\#^*$ is a well-defined operator extending $\tilde{\grad}$. Note that it can be shown that \[\dom(\grad_\#)=H^1_\#(\Omega)\coloneqq \{\phi|_\Omega; \phi\in H^1_{\textnormal{loc}}(\mathbb{R}^2), \phi(\cdot)=\phi(\cdot+k)\quad (k\in\mathbb{Z}^2)\}.\] 

Next, let $s_0,s_1\colon \mathbb{R}^2\to \mathbb{C}$ be measurable, bounded, $(0,1)^2$-periodic functions satisfying $s_0(x)=s_0(x)^*\geq 0$ for all $x\in \mathbb{R}^2$ and
\[
   \rho_0 s_0(x)+\Re s_1(x)\geq c
\]
for some $\rho_0\geq 0$ and $c>0$ and all $x\in \mathbb{R}^2$.

We define $M_0 \in L(L^2(\Omega)^3)$ by
\[
   M_0 (\phi)_{j\in\{1,2,3\}} \coloneqq \begin{pmatrix} s_0 & 0 \\ 0 & 1 
   \end{pmatrix}\begin{pmatrix} \phi_1 \\ (\phi_j)_{j\in\{2,3\}}\end{pmatrix} \coloneqq \begin{pmatrix} (\Omega\ni x\mapsto s_0(x)\phi_1(x)) \\ (\phi_j)_{j\in\{2,3\}}\end{pmatrix},
\]
and $M_1$ similarly replacing $s_0$ by $s_1$. Note that we have 
\[
   \rho_0 M_0 + \Re M_1 \geq c
\]
in the sense of positive definiteness; furthermore $M_0$ is selfadjoint. A straight forward application of \cite[Solution Theory]{Picard2009} leads to the following result. We recall that $\partial_t$ is the distributional derivative with respect to the first variable in the space
\[
  L^2_{\rho}(H)\coloneqq \{ f\in L^2_{\textnormal{loc}}(\R;H); \int_\R \|f(t)\|^2_H \exp(-2\rho t)dt<\infty \}
\]
with maximal domain $H^1_\rho(H)$, that is,
\[
   \partial_t \colon H^1_\rho(H) \subseteq L^2_\rho(H)\to L^2_\rho(H), \phi\mapsto \phi'.
\]
It will be obvious from the context, which $\rho$ and which Hilbert space $H$ is chosen. In the next theorem, we have $H=L^2(\Omega)^3$. 
\begin{theorem}[\cite{Picard2009}]\label{thm:st} Let $\rho\geq\rho_0$. Then
\[
   \mathcal{S}\coloneqq   \overline{    \partial_t M_0 + M_1 + \begin{pmatrix} 0 & \dive_\# \\ \grad_\# & 0 \end{pmatrix}}^{-1} \in L(L^2_{\rho}(L^2(\Omega)^3))
\]
and $\|\mathcal{S}\|\leq 1/c$.
\end{theorem}

\begin{remark}\label{rem:st} Note that it can be shown (\cite[Remark 2.3]{W16_SH}) that if $F\in \dom(\partial_t)$, then $\mathcal{S}F\in \dom(\partial_t)\cap \dom(\begin{pmatrix} 0 & \dive_\# \\ \grad_\# & 0 \end{pmatrix})$. Moreover, let $F=(f,0)\in L^2_{\rho}(L^2(\Omega)\oplus L^2(\Omega)^2)$. Then $U=(u,v)=\mathcal{S}F$ satisfies the following two equations
\begin{align*}
 \partial_t s_0 u + s_1 u + \dive_\# v &=f \\
 \partial_t v &= - \grad_\# u.
\end{align*}
Substituting the second equation into the first one, we obtain
\[
    \partial_t^2 s_0 u + \partial_t s_1 u + \dive_\#\grad_\# u = \partial_t f,
\]
which is a damped wave equation, if $s_0>0$ everywhere. The conditions imposed on $s_0 $ and $s_1$, however, also allow for regions, where $s_0 =0$ (or $s_0=0$ entirely). On these regions, the equation is a heat-type equation. If there are regions where either $s_0$ or $s_1$ vanish (but not both on the same region), the resulting equation is of mixed type. We emphasise, that transmission conditions are not necessary for the formulation of the equations but are rather a consequence of $U$ being a solution to the equation; see also \cite[Remark 3.2]{W16_SH}.
\end{remark}

Next, we aim to study the limit behaviour of $ \mathcal{S}_N$, which is given as $\mathcal{S}$ but with $s_0(N\cdot)$ and $s_1(N\cdot)$ respectively replacing $s_0$ and $s_1$. In particular, our aim is to establish the following theorem. For this, we define
\[
    H_\rho^k(H)\coloneqq  \dom(\partial_t^k)
\]
endowed with the graph norm of $\partial_t^k$ acting as an operator from $L^2_\rho(H)$ into itself. It can be shown that given $\rho>\rho_0$ that $\partial_t$ is continuously invertible in $L^2_\rho(H)$; so that $u\mapsto \|\partial_t^ku\|$ is equivalent to the graph norm on $H^k_\rho(H)$. 
\begin{theorem}\label{thm:qht} Let $\rho>\rho_0$. There exists $\kappa\geq 0$ such that for all $N\in \mathbb{N}$ and $f\in H_\rho^2(L^2(\Omega))$ we have
\begin{multline*}
   \left\| \left(\partial_t \begin{pmatrix} s_0 (N\cdot) & 0 \\ 0 & 1\end{pmatrix}+\begin{pmatrix} s_1 (N\cdot) & 0 \\ 0 & 0 \end{pmatrix} + \begin{pmatrix} 0 & \dive_\# \\ \grad_\# & 0\end{pmatrix}\right)^{-1}\begin{pmatrix}f \\ 0\end{pmatrix} \right. \\- \left.
   \left(\partial_t \begin{pmatrix} \langle s_0\rangle & 0 \\ 0 & 1\end{pmatrix}+\begin{pmatrix} \langle s_1\rangle & 0 \\ 0 & 0 \end{pmatrix} + \begin{pmatrix} 0 & \dive_\# \\ \grad_\# & 0\end{pmatrix}\right)^{-1}\begin{pmatrix}f \\ 0\end{pmatrix} \right\|_{L_\rho^2(L^2(\Omega)^3)} \\
   \leq \frac{\kappa}{N} \| \partial_t^2 f\|_{L_\rho^2(L^2(\Omega))},
\end{multline*}where
\[
     \langle s_j\rangle\coloneqq \int_\Omega s_j(x) dx \quad(j\in\{0,1\}).
\]
\end{theorem}
In order to prove this theorem, we need to introduce the Fourier--Laplace transformation: Let $H$ be a Hilbert space. For $\phi\in C_c(\mathbb{R};H)$ we define
\[
    \mathcal{L}_\rho \phi(\xi)\coloneqq \frac{1}{\sqrt{2\pi}}\int_{\mathbb{R}} \phi(t)\exp(-it\xi-\rho t)dt.
\]
A variant of Plancherel's theorem yields that $\mathcal{L}_\rho$ extends to a unitary operator from $L_\rho^2(H)$ into $L^2(H)$. A remarkable property of $\mathcal{L}_\rho$ is that
\[
    \partial_t =\mathcal{L}^*_\rho (im+\rho)\mathcal{L}_\rho,
\]
where $m$ is the multiplication by argument operator in $L^2(\mathbb{R};H)$ with maximal domain; see \cite[Corollary 2.5]{KPSTW14_OD}. Thus, applying the Fourier--Laplace transformation to the norms on either side of the inequality in Theorem \ref{thm:qht}, we deduce that it suffices to show that there exists $\kappa\geq0$ such that for all $N\in \mathbb{N}$, $z\in \mathbb{C}_{\Re\geq \rho}$ and $f\in L^2(\Omega)$ we have
\begin{align}\notag
  & \left\| \left(z \begin{pmatrix} s_0 (N\cdot) & 0 \\ 0 & 1\end{pmatrix}+\begin{pmatrix} s_1 (N\cdot) & 0 \\ 0 & 0 \end{pmatrix} + \begin{pmatrix} 0 & \dive_\# \\ \grad_\# & 0\end{pmatrix}\right)^{-1}\begin{pmatrix}f \\ 0\end{pmatrix} \right. \\ &\quad\quad- \left.
   \left(z \begin{pmatrix} \langle s_0\rangle & 0 \\ 0 & 1\end{pmatrix}+\begin{pmatrix} \langle s_1\rangle & 0 \\ 0 & 0 \end{pmatrix} + \begin{pmatrix} 0 & \dive_\# \\ \grad_\# & 0\end{pmatrix}\right)^{-1}\begin{pmatrix}f \\ 0\end{pmatrix} \right\|_{L^2(\Omega)^3} \label{eq:crucial}\\
 &\quad\quad\quad\quad   \leq \frac{\kappa}{N} \| z^2 f\|_{L^2(\Omega)}.\notag
\end{align} 
This inequality will be shown using the results of \cite{CW17}. For this we need some auxiliary statements.
\begin{lemma}\label{lem:ranclosed} The space $\ran(\gradp)$ is closed on $L^2(\Omega)^2$.
\end{lemma}
\begin{proof}
 Since $\Omega$ has continuous boundary, we get that $H^1(\Omega)$ embeds compactly into $L^2(\Omega)$. Since $\gradp\subseteq \grad$, where $\grad\colon H^1(\Omega)\subseteq L^2(\Omega)\to L^2(\Omega)^2$ is the distributional gradient and $\gradp$ is closed, we obtain that $H^1_\#(\Omega)$ is compactly embedded into $L^2(\Omega)$, as well. It is now standard to show that $\ran(\gradp)\subseteq L^2(\Omega)^2$ is closed, see e.g.~\cite[Lemma 4.1(b)]{EGW17_D2N}.
\end{proof}
Using Lemma \ref{lem:ranclosed}, we define 
\[
\iota\colon \ran(\gradp)\hookrightarrow L^2(\Omega)^2, \phi\mapsto \phi
\]
and obtain that 
\[
   \iota^*\colon L^2(\Omega)^2\to \ran(\gradp)
\]
is the (surjective) orthogonal projection according to the decomposition $L^2(\Omega)^2 = \ker(\divep)\oplus\ran(\gradp)$.
\begin{proposition}[{{\cite[Proposition 3.8]{CW17}}}]\label{prop:1storder} Let $f\in L^2(\Omega)$. Then the following conditions are equivalent:
\begin{enumerate}
\item $u\in \dom(\divep\gradp)$ satisfies
\[
    -\divep\gradp u + z^2s_0 + zs_1 u = f
\]
\item $u\in \dom(\gradp)$ and $q\in \dom(\divep)$ satisfy
\[
  \left(    \begin{pmatrix}
      z s_0 + s_1 & 0 \\ 0 & z 
    \end{pmatrix} + \begin{pmatrix} 0 & \divep \\ \gradp & 0 \end{pmatrix}\right)\begin{pmatrix} u \\ q\end{pmatrix} = \begin{pmatrix} z^{-1}f \\ 0 \end{pmatrix}.
\]
\item $u\in \dom(\gradp)$ and $q\in \dom(\divep)\cap \ran(\gradp)$ satisfy
\[
  \left(    \begin{pmatrix}
      z s_0 + s_1 & 0 \\ 0 & z 
    \end{pmatrix} + \begin{pmatrix} 0 & \divep \iota \\ \iota^* \gradp & 0 \end{pmatrix}\right)\begin{pmatrix} u \\ q\end{pmatrix} = \begin{pmatrix} z^{-1}f \\ 0 \end{pmatrix}.
\]
\end{enumerate}
\end{proposition}
\begin{proof}
The equivalence of 1 and 2 follows from \cite[Proposition 3.8]{CW17} by multiplying 1 by $z^{-1}$ and by putting $\varepsilon=1$, $\theta=0$, $n=1$, $s=zs_0+s_1$ and $a=z^{-1}$ in  \cite[Proposition 3.8]{CW17}. The implication from 3 to 1 follows upon realising that $\divep\gradp=\divep\iota\iota^*\gradp$. Thus, it remains to establish that 2 is sufficient for 3. For this implication, however, note that the second equation in 2, implies that $zq\in \ran(\gradp)$ and, hence, $q\in \ran(\gradp)$. Therefore, $zq=\iota\iota^*zq=z\iota\iota^*q$.
\end{proof}

Next, we introduce the Floquet--Bloch or Gelfand transformation:

\begin{definition}
 Let $N\in \mathbb{N}$, $f\colon \mathbb{R}^2\to \mathbb{C}$. Then define
 \begin{multline*}
      \mathcal{V}_N f(\theta,y)\coloneqq \frac{1}{N}\sum_{k\in \{0,\ldots,N-1\}^2} f(y+k)e^{-i\theta\cdot k}\\(y\in [0,1)^2, \theta \in \{2\pi k /N; k\in \{0,\ldots,N-1\}^2\})
 \end{multline*}
 and for $f\in L^2(0,1)$
 \[
     T_N f \coloneqq \frac1N f\left(\frac{\cdot}{N}\right).
 \]
\end{definition}

As in \cite{FrW17} one can show the following result:
\begin{theorem}\label{thm:gt}
(a) The mapping $V_N \colon L^2_\#((0,N)^2)\to L^2(0,1)^{N^2}$ given by
\[
   f\mapsto (\mathcal{V}_N f( 2\pi k/N, \cdot))_{k\in \{0,\ldots,N-1\}^2}
\]
is unitary, where $L^2_\#((0,N)^2)$ denotes the set of $(0,N)^2$-periodic $L^2_{\textnormal{loc}}(\mathbb{R}^2)$ functions endowed with the scalar product from $L^2((0,N)^2)$.

(b) The mapping $G_N\coloneqq V_N\mathcal{T}_N$ is unitary.
\end{theorem}

The mapping $G_N$ in the previous theorem is also called the Floquet--Bloch or Gelfand transformation. With this tranformation at hand, we are in the position to transform the inequality in \eqref{eq:crucial} into an equivalent form such that \cite[Section 2]{CW17} is applicable. The reason is the following representation:

\begin{proposition}[{{\cite{CW17} and \cite{FrW17}}}] Let $N\in \mathbb{N}$, $k\in \{0,\ldots, N-1\}^2$ and $\theta\coloneqq 2\pi k/N$, $f\in L^2(\Omega)$. Then we have
\begin{multline*}
\left(G_N\left(    \begin{pmatrix} zs_0(N\cdot) + s_1(N\cdot) & 0 \\ 0 & z \end{pmatrix} + \begin{pmatrix} 0 & \divep\iota \\ \iota^* \gradp & 0\end{pmatrix}\right)^{-1}\begin{pmatrix} f \\ 0 \end{pmatrix} G_N^*\right)_k \\
= \left(\begin{pmatrix} zs_0(\cdot) + s_1(\cdot) & 0 \\ 0 & z \end{pmatrix} + \frac{1}{N}\begin{pmatrix} 0 & \dive_\theta\iota_{\theta} \\ \iota_{\theta}^*\grad_\theta & 0\end{pmatrix}\right)^{-1}\begin{pmatrix} (G_N f G_N^*)_k \\ 0 \end{pmatrix},
\end{multline*}
where $\dive_\theta$ and $\grad_\theta$ as well as $\iota_\theta$ are given as in \cite[Section 3]{CW17}.
\end{proposition}
\begin{proof}
Let $(u,q)\coloneqq \left(    \begin{pmatrix} zs_0(N\cdot) + s_1(N\cdot) & 0 \\ 0 & z \end{pmatrix} + \begin{pmatrix} 0 & \divep\iota \\ \iota^* \gradp & 0\end{pmatrix}\right)^{-1}\begin{pmatrix} f \\ 0 \end{pmatrix}$.  By Pro\-position \ref{prop:1storder}, we have that 
\[
    -\divep\gradp u +z^2s_0(\cdot)u+ z s_1(\cdot)u = zf, \quad zq= -\gradp u.
\]
Then, by the argument just after \cite[Proposition 3.5]{CW17} (use an adapted version of \cite[Proposition 3.5]{CW17}, where the Gelfand transform used there is replaced by the discrete version introduced here), it follows that $u_k\coloneqq (G_N uG_N^*)_k$ satisfies \[-\frac{1}{N^2}\dive_\theta\grad_\theta u_k +z^2s_0(\cdot)u_k+ z s_1(\cdot)u_k = z (G_N f G_N^*)_k\eqqcolon zf_k.\] Applying $G_N$ to $zq=-\gradp u$, we obtain
\[
    z (G_N q)_k = -\grad_\theta (G_N u)_k = - \iota_\theta^*\grad_\theta (G_N u)_k,
\]
which yields the assertion.
\end{proof}

Now, along the lines of \cite[Section 3]{CW17} it is possible to show the following result, which eventually implies Theorem \ref{thm:qht}.

\begin{theorem}[{{\cite[Proof of Theorem 3.1; Eq (14)]{CW17}}}] There exists $\kappa\geq 0$ such that for all $N\in\mathbb{N}$, $k \in \{0,\ldots,N-1\}^2$ with $\theta=2\pi k/N$ and $f\in L^2(\Omega)$ we have
\begin{multline*}
  \left\| \left(  \begin{pmatrix} zs_0(\cdot) + s_1(\cdot) & 0 \\ 0 & z \end{pmatrix} + \frac{1}{N}\begin{pmatrix} 0 & \dive_\theta\iota_{\theta} \\ \iota_{\theta}\grad_\theta & 0\end{pmatrix} \right)^{-1} \right. \\
  -  \left. \left(  \begin{pmatrix} z\langle s_0\rangle  + \langle s_1\rangle & 0 \\ 0 & z \end{pmatrix} + \frac{1}{N}\begin{pmatrix} 0 & \dive_\theta\iota_{\theta} \\ \iota_{\theta}\grad_\theta & 0\end{pmatrix} \right)^{-1} \right\|\leq \frac{\kappa}{N}\|z^2f\|_{L^2(\Omega)}.
\end{multline*}
\end{theorem}

With this theorem, the assertion of Theorem \ref{thm:qht} follows upon applying the inverse Gelfand transformation first and afterwards the inverse Fourier--Laplace transformation; see also \cite[Proof of Theorem 3.10]{FrW17} for the precise argument.

\section{Numerical method}\label{sec:numerics}
In this whole section, we address solving the equation
\begin{equation}\label{eq:problemnum}
  \left(   \partial_t M_0 + M_1 + \begin{pmatrix} 0 &\divep \\ \gradp & 0 \end{pmatrix}\right)U = F
\end{equation}
The analytical results (Theorem \ref{thm:st} and Remark \ref{rem:st}) state that given $F\in H_\rho^1(H)$, the solution $U$ of \eqref{eq:problemnum} yields 
$U\in L^2_\rho(H^1_\#(\Omega)\times H(\dive_\#,\Omega))$ such that $M_0U\in H^1_\rho(H)$ with
$H=L^2(\Omega)^3$. We will use a discontinuous Galerkin method in time and
a conforming Galerkin method in space. For that let $0=t_0<t_1<\dots<t_M=T$ be a mesh for the 
time interval $[0,T]$ using $M$ equidistant intervals $I_m=(t_{m-1},t_{m})$ of length 
$\tau=t_{m}-t_{m-1}=\frac{T}{M}$,\,$m\in\{1,\dots,M\}$. The method could also be defined on a non-uniform
mesh in time with the obvious changes. For the discretisation of $\bar\Omega=[0,1]^2$ we use an 
equidistant tensor-product mesh with mesh-cells $K_{ij}=(x_{i-1},x_i)\times(y_{j-1},y_j)$, where
$x_i=\frac{1}{N}$, $i\in\{0,\dots,N\}$ and $y_j=\frac{1}{N}$, $j\in\{0,\dots,N\}$. Again a non-equidistant
tensor product mesh with different mesh-sizes in the different dimensions is also possible.

We will approximate $U=(u,v)$ using piecewise polynomials, globally discontinuous in time and 
piecewise polynomials, globally continuous ($H^1$-conforming) in space for $u$ and globally $H(\dive)$-conforming for $v$.
Thus our discrete space is given by
\[
  \U^{h,\tau}
  \coloneqq \big\{
        U\in H_\rho([0,T];H):\,
        U|_{I_m},\!\in\PS_q(I_m,\V_u(\Omega)\times \V_v(\Omega)),
        m\!\in\!\{1,\dots,M\}
      \big\},
\]
where the spatial spaces are
\begin{align*}
  \V_u(\Omega)
     & \coloneqq \left\{u\in H^1_{\#}(\Omega):\,u|_{K_{ij}}\in\QS_p(K_{ij}),\,0\leq i,j\leq N\right\},\\
  \V_v(\Omega)
     & \coloneqq \left\{v\in H_{\#}(\dive,\Omega):\,v|_{K_{ij}}\in \RT_{p-1}(K_{ij}),\,0\leq i,j\leq N\right\}.    
\end{align*}
Here, $\PS_q(I_m,H)$ is the space of polynomials of degree up to $q$ on the interval $I_m$ with values in $H$
and $\QS_p(K_{ij})$ is the space of polynomials with total degree up to $p$ on the cell $K_{ij}\subseteq\Omega$. 
Furthermore, $\RT_{p-1}(K_{ij})$ is the Raviart--Thomas space on $K_{ij}$, defined by
\[
  \RT_{p-1}(K_{ij})=(\QS_{p-1}(K_{ij}))^n+\vx\QS_{p-1}(K_{ij}).
\]
Note that
\begin{align*}
  (\QS_{p-1}(K_{ij}))^n\subset \RT_{p-1}(K_{ij})&\subset(\QS_p(K_{ij}))^n,\\
  \dive(\RT_{p-1}(K_{ij}))&\subset \QS_{p-1}(K_{ij})
  \quad\text{and}\quad \\
  \RT_{p-1}(K_{ij})\cdot\vn|_{\partial K_{ij}}&\subset \PS_{p-1}(\partial K_{ij}).
\end{align*}
Finally, the ``$\#$'' denotes periodic boundary conditions. This means, that $w\in \V_u(\Omega)$ fulfils
\[
  w(0,\zeta)=w(1,\zeta),\,w(\zeta,0)=w(\zeta,1),\quad\text{for any }\zeta\in[0,1]
\]
and $w\in \V_v(\Omega)$ fulfils using the outer normal $\vn$ on $\partial\Omega$
\[
  (\vn\cdot w)(0,\zeta)=-(\vn\cdot w)(1,\zeta),\,(\vn\cdot w)(\zeta,0)=-(\vn\cdot w)(\zeta,1),\quad\text{for any }\zeta\in[0,1].
\]

With these notions at hand, we can now properly specify the numerical method.
For any given right-hand side $F\in \U^{h,\tau}$ and initial condition $x_0\in H$, 
find $\U\in\U^{h,\tau}$, such that for all $\Phi\in \U^{h,\tau}$ and $m\in\{1,2,\dots,M\}$ 
it holds
\begin{equation}\label{eq:discr_quad_form}
  \Qmr{(\partial_t M_0+M_1+A)\U,\Phi}
    +\scp{M_0 \jump{\U}_{m-1}^{x_0},{\Phi}^+_{m-1}}
    =\Qmr{ F,\Phi }.
\end{equation}
      
Here, we denote by 
\[
  \jump{\U}_{m-1}^{x_0}:=
  \begin{cases}
    \U(t_{m-1}+)-U(t_{m-1}-),& m\in\{2,\ldots,M\}\\ 
    \U(t_0+)-x_0,& m=1,
  \end{cases} 
\]
the jump at $t_{m-1}$, by $\Phi^+_{m-1}:= \Phi(t_{m-1}+)$ the right-sided trace and by 
\[
  \Qmr{a,b}:=\frac{\tau_m}{2} \sum_{i=0}^q {\omega}^m_i \scp{a(\tmi),b(\tmi)}
\]
a right-sided weighted Gau\ss--Radau quadrature formula on $I_m$ approximating 
\[
  \scprm{a,b} :=
    \intop_{t_{m-1}}^{t_m} 
    \scp{a(t),b(t)} \exp(-2\rho(t-t_{m-1})) \mathrm{d} t,
\]
see \cite{FrTW16} for further details. 

We can cite the convergence results from \cite{FrTW16} which were for Dirichlet boundary conditions.
The proof needs only marginal modifications to hold for the periodic case too. We introduce two
measures for the error. The first one measures the error in an $L^\infty$-$L^2$ sense with
\[
  E^2_{\sup}(a):=\sup_{t\in[0,T]}\langle M_0 a(t),a(t)\rangle,
\]
while the second is a discrete version of the $L^2_\rho(H)$-norm, given by
\[
  E^2_Q(a):=\e^{2\rho T}\sum_{m=1}^M\Qmr{a,a}\e^{-2\rho t_{m-1}}.
\]
Note that $E_Q(a)=\|a\|$ for $a\in\U^{h,\tau}$.

\begin{theorem}\label{theorem:conv_numer}
  We assume for the solution $U$ of Example \eqref{eq:problemnum} the 
  regularity     
  \[
    U\in H_\rho^{1}(H_{\#}^p(\Omega)\times H_{\#}^p(\Omega)^2)\cap 
         H_\rho^{q+3}(L^2(\Omega)\times L^2(\Omega)^2) 
  \]
  as well as 
  \[
    AU\in L^2_\rho( H_{\#}^p(\Omega)\times H_{\#}^p(\Omega)^2).
  \]
  Then we have for the error of the numerical solution $U^{h,\tau}$ of 
  \eqref{eq:discr_quad_form} with a generic constant $C$
  \[
    E^2_{\sup}(U-U^{h,\tau})+
    E^2_Q(U-U^{h,\tau})
    \leq C \e^{2\rho T}(\tau^{2(q+1)} + T h^{2p}).
  \]
\end{theorem}
Note that the spatial regularity is only needed in each cell $K_{ij}$ of the spatial mesh
as local interpolation error estimates are used.

\section{Numerical study}\label{sec:simulations}
All computations were done in $\mathbb{SOFE}$ (\url{https://github.com/SOFE-Developers/SOFE}), 
a finite element suite for Matlab/Octave.

For our numerical study let us assume an equidistant rectangular background mesh 
covering $\Omega$ with nodes $(x_i=\frac{i}{N},y_j=\frac{j}{N})$,\,$i,j\in\{0,\dots,N\}$ for an even number 
$N\in\N$. This background mesh will be used in defining the oscillating coefficients.

Our rough coefficient problem is given by
\begin{gather}\label{eq:prob1}
  \left( 
    \partial_t
    \begin{pmatrix}
      \epsilon_N & 0\\
            0 & 1
    \end{pmatrix}
    +\begin{pmatrix}
      1-\epsilon_N & 0\\
              0 & 0
    \end{pmatrix}
    +\begin{pmatrix}
      0 & \divep \\
      \gradp & 0
    \end{pmatrix}
  \right)U_N
  =\begin{pmatrix}
      f\\
      0
    \end{pmatrix},
\end{gather}
where the coefficient function $\epsilon_N$ is defined as
\[
  \epsilon_N(x,y)\coloneqq \begin{cases}
                      1,& \exists i,j\in\N_0:(x,y)\in(x_{i},x_{i+1})\times(y_{j},y_{j+1})\text{ and }i+j\text{ is even},\\
                      0,& \text{otherwise.}
                    \end{cases}.
\]
The corresponding homogenised problem is then 
\begin{gather}\label{eq:prob2}
  \left( 
    \partial_t
    \begin{pmatrix}
      \frac{1}{2} & 0\\
            0 & 1
    \end{pmatrix}
    +\begin{pmatrix}
      \frac{1}{2} & 0\\
              0 & 0
    \end{pmatrix}
    +\begin{pmatrix}
      0 & \divep \\
      \gradp & 0
    \end{pmatrix}
  \right)U_{\textnormal{hom}}
  =\begin{pmatrix}
      f\\
      0
    \end{pmatrix}.
\end{gather}
The 
theoretical results of Sections~\ref{sec:analysis} and \ref{sec:numerics} provide the following expected convergence 
behaviour
\begin{align*}
  \|U_N-U_{\textnormal{hom}}\|_{H^1_\rho(\R,H)}&\leq C N^{-1},\\
  E_Q(U_{\textnormal{hom}}^{h,\tau}-U_{\textnormal{hom}})&\leq C (h^{p}+\tau^{q+1}),\quad 
  E_Q(U_N^{h,\tau}-U_N)\leq C (h^{p}+\tau^{q+1})
\end{align*}
for smooth solutions $U_{\textnormal{hom}}$ and $U_{N}$. In general we cannot expect the solutions to be very smooth.
Thus, for our experiments we only chose a polynomial order $p=2$ in space and $q=1$ in time. 
Setting furthermore $h=\tau=1/(2N)$ we combine the above expected estimates 
and obtain
\begin{align*}
  E_Q(U_N^{h,\tau}-U_{\textnormal{hom}})
    &\leq E_Q(U_N^{h,\tau}-U_N)+E_Q(U_N-U_{\textnormal{hom}})\\
    &\leq E_Q(U_N^{h,\tau}-U_N)+C\|U_N-U\|_{H^1_\rho(\R,H)}
    \leq C N^{-1},
\end{align*}
where the second inequality comes from Sobolev's embedding theorem (see e.g.~\cite[Lemma 5.2]{KPSTW14_OD})

Let us finalise the definition of our problem by setting the right-hand side
\[
  f(t,x,y)=\begin{cases}
            1, & t\in(0,1)\text{ and }\max\{|2x-1|,|2y-1|\}\leq \frac{1}{4},\\
            0,& \text{otherwise.}
           \end{cases}\]
Thus $f$ is one in the time-space cube $(0,1)\times[1/4,3/4]^2$ and otherwise zero.
Figure~\ref{fig:sol}
\begin{figure}[bt]
\includegraphics[width=0.24\textwidth]{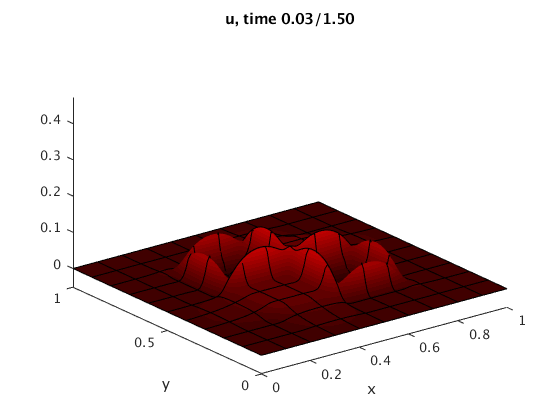}
\includegraphics[width=0.24\textwidth]{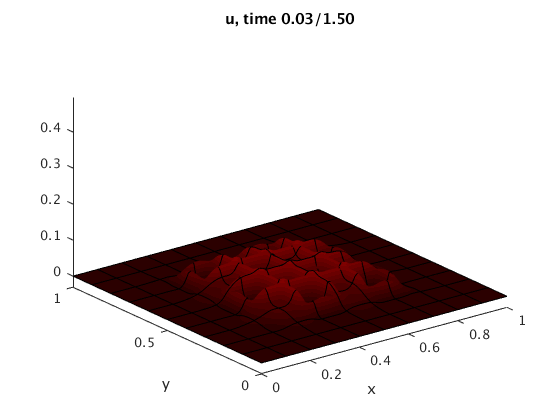}
\includegraphics[width=0.24\textwidth]{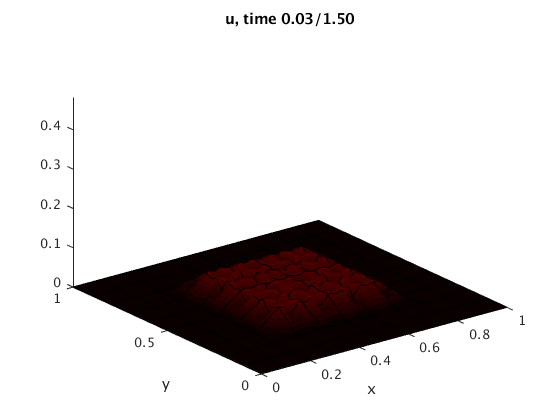}
\includegraphics[width=0.24\textwidth]{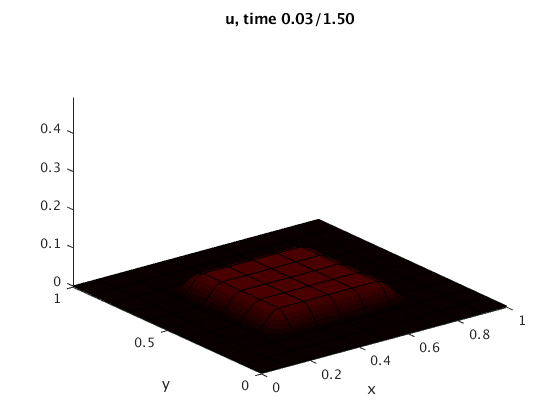}\\
\includegraphics[width=0.24\textwidth]{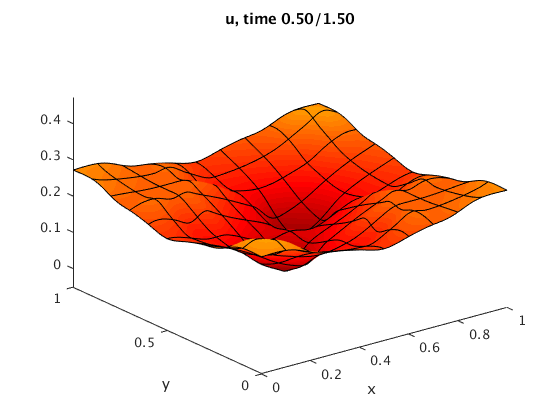}
\includegraphics[width=0.24\textwidth]{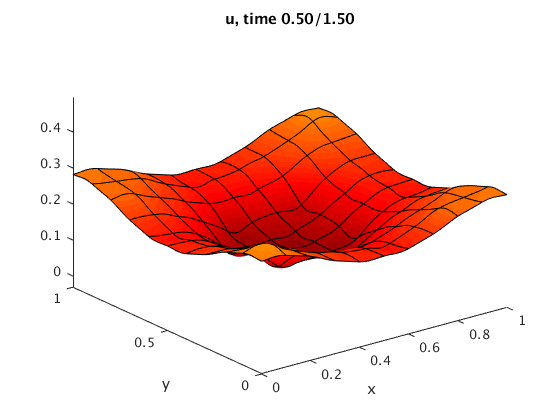}
\includegraphics[width=0.24\textwidth]{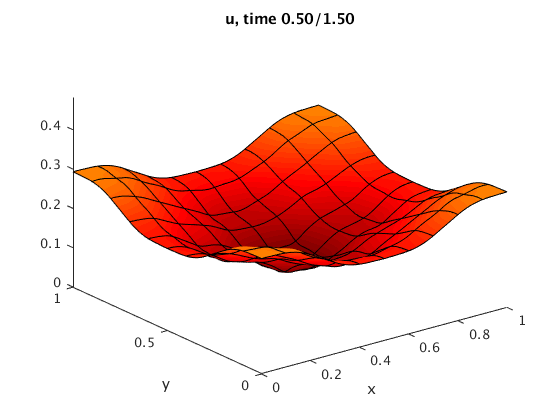}
\includegraphics[width=0.24\textwidth]{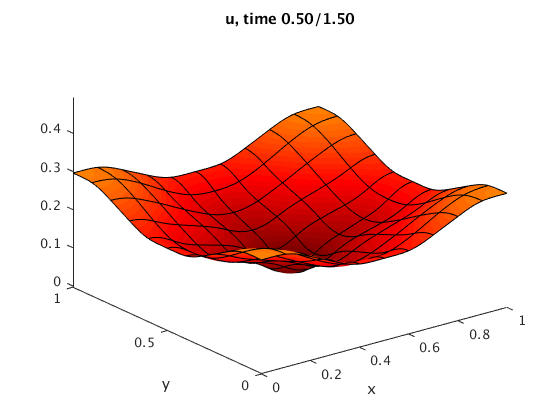}\\
\includegraphics[width=0.24\textwidth]{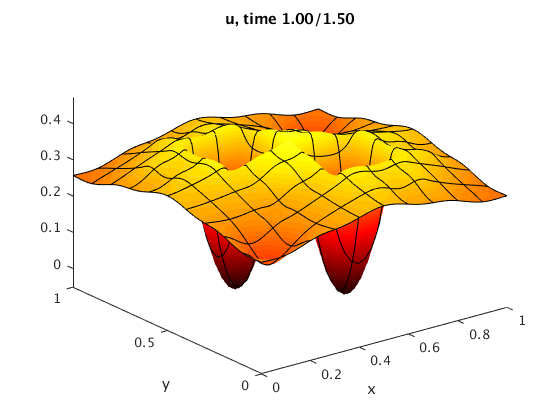}
\includegraphics[width=0.24\textwidth]{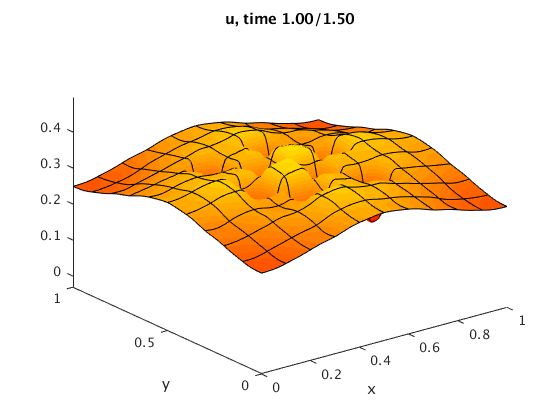}
\includegraphics[width=0.24\textwidth]{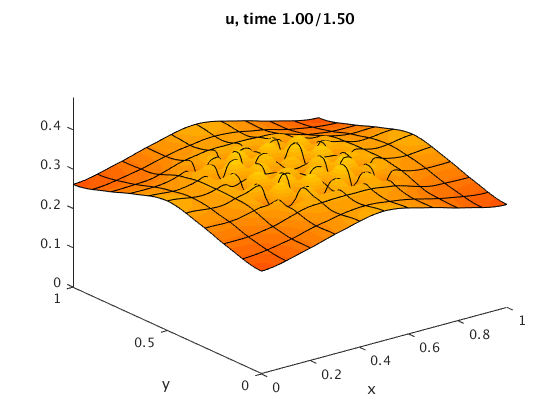}
\includegraphics[width=0.24\textwidth]{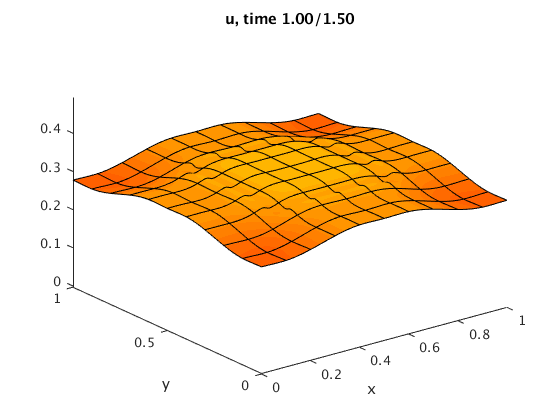}\\
\includegraphics[width=0.24\textwidth]{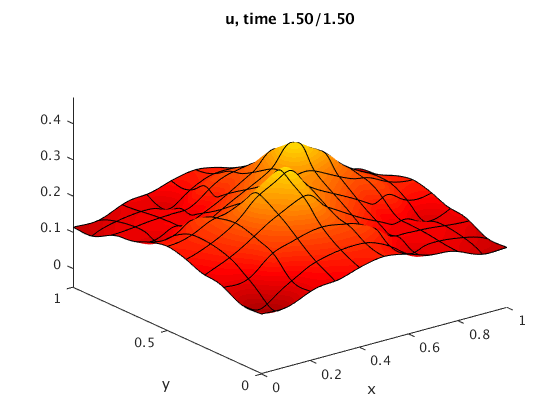}
\includegraphics[width=0.24\textwidth]{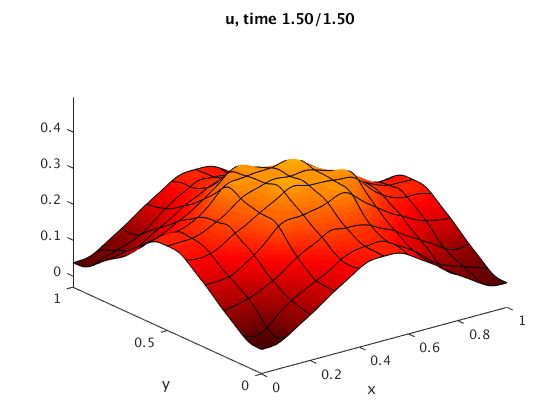}
\includegraphics[width=0.24\textwidth]{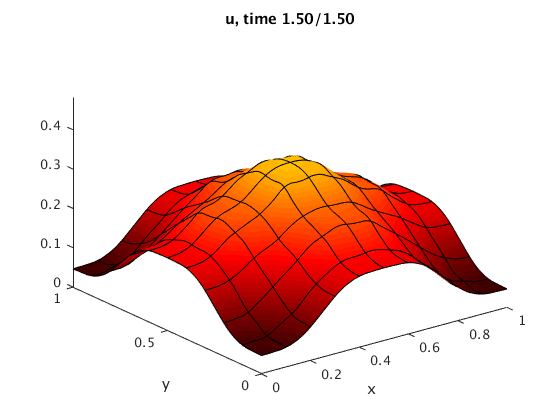}
\includegraphics[width=0.24\textwidth]{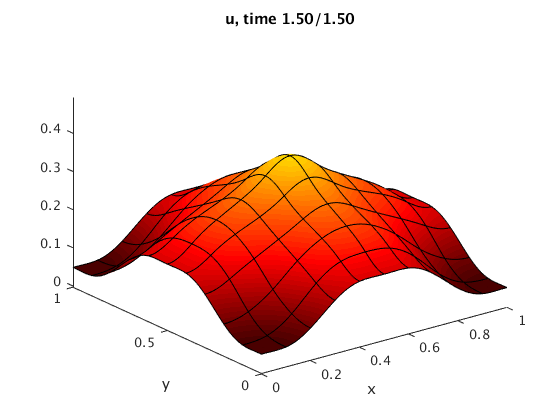}
\caption{Solutions $U_4$, $U_8$, $U_{16}$ and $U_{\textnormal{hom}}$ (left to right) 
         at $t=0.025,\,0.5,\,1$ and $1.5$
         (top to bottom)}
\label{fig:sol}
\end{figure}
shows (numerical approximations of) the solutions $U_4$, $U_8$, $U_{16}$ and $U_{\textnormal{hom}}$ at different times.
In the first row the rough coefficients can be seen quite nicely, while the solution 
becomes smooth very quickly (lower rows). Furthermore, already for a very coarse background mesh 
of $N=16$ the solutions $U_N$ and $U_{\textnormal{hom}}$ are very similar. This visualises the 
homogenisation process.

In Table~\ref{tab:prob1}
    \begin{table}[tb]
       \caption{Convergence results for $\tilde U_N-U_N^{h,\tau}$ and $\tilde U_{\textnormal{hom}}-U_N^{h,\tau}$ of problem \eqref{eq:prob1} using $h=\tau=\frac{1}{2N}$
                \label{tab:prob1}}
       \begin{center}   
        \begin{tabular}{      r@{\qquad}
                        l@{~~}l@{\qquad}
                        l@{~~}l@{\qquad}
                        l@{~~}l@{\qquad}
                        l@{~~}l}
        \toprule
           $N$ & 
           \multicolumn{2}{c}{$E_{\sup}(\tilde U_N-U_N^{h,\tau})$    } &
           \multicolumn{2}{c}{$E_Q(\tilde U_N-U_N^{h,\tau})$         } & 
           \multicolumn{2}{c}{$E_{\sup}(\tilde U_{\textnormal{hom}}-U_N^{h,\tau})$} &
           \multicolumn{2}{c}{$E_Q(\tilde U_{\textnormal{hom}}-U_N^{h,\tau})$     }\\
        \midrule
            2 & 5.046e-02 &      & 1.336e-02 &      & 7.175e-02 &      & 2.778e-02 &     \\
            4 & 2.346e-02 & 1.11 & 6.692e-03 & 1.00 & 4.391e-02 & 0.71 & 1.969e-02 & 0.50\\
            8 & 1.171e-02 & 1.00 & 3.165e-03 & 1.08 & 2.256e-02 & 0.96 & 8.802e-03 & 1.16\\
           16 & 6.063e-03 & 0.95 & 1.507e-03 & 1.07 & 1.038e-02 & 1.12 & 4.186e-03 & 1.07\\
           32 & 3.172e-03 & 0.93 & 6.633e-04 & 1.18 & 5.081e-03 & 1.03 & 2.005e-03 & 1.06\\
           64 & 1.590e-03 & 1.00 & 3.012e-04 & 1.14 & 2.383e-03 & 1.09 & 9.445e-04 & 1.09\\ 
        \bottomrule
        \end{tabular}
       \end{center}
    \end{table}
we see the results for a simulation using polynomial degrees $p=q+1=2$. As no exact solutions to \eqref{eq:prob1} and \eqref{eq:prob2}
are known, we use reference solutions $\tilde U_{N}$ and $\tilde U_{\textnormal{hom}}$ computed with polynomial degree $p=3$ on a mesh with 
256 cells in each space dimension and 384 cells in time dimension. The reference solution mesh is therefore twice 
as fine as the finest one used in the simulation.

Note that we also provided the experimental
orders of convergence (eoc), calculated for errors $E_n$ and $E_{2n}$ by
\[
  \textnormal{eoc}_n=\frac{\ln\frac{E_n}{E_{2n}}}{\ln 2}.
\]
We observe a first order convergence 
of the numerical solution $U_N^{h,\tau}$ towards $U_N$ and towards $U_{\textnormal{hom}}$. While the second result
confirms the reasoning at the beginning of this section, the first directs to a non-smoothness of the solution
as otherwise we would obtain a second order convergence, see Theorem~\ref{theorem:conv_numer}. 
Considering the oscillating coefficients and discontinuous $f$ this reduction is to be expected.

%

\bibliographystyle{plain}

\end{document}